\documentclass[12pt,reqno]{amsart}

\usepackage{amsfonts, amsthm, amsmath}
\allowdisplaybreaks[4]

\usepackage{rotating}

\usepackage{tikz}

\usepackage{graphics}

\usepackage{amssymb}

\usepackage{amscd}

\usepackage[latin2]{inputenc}

\usepackage{t1enc}

\usepackage[mathscr]{eucal}

\usepackage{indentfirst}

\usepackage{graphicx}

\usepackage{graphics}

\usepackage{pict2e}

\usepackage{mathrsfs}

\usepackage{enumerate}
\usepackage[pagebackref]{hyperref}
\hypersetup{backref, pagebackref, colorlinks=true}
\usepackage{cite}
\usepackage{color}
\usepackage{epic}
\usepackage{hyperref} %this gives clickable references, which is nice
\numberwithin{equation}{section}
\theoremstyle{plain}

\newtheorem{theorem}{Theorem}[section]

\newtheorem{lemma}[theorem]{Lemma}

\newtheorem{conjecture}[theorem]{Conjecture}

\theoremstyle{definition}

\newtheorem{remark}[theorem]{Remark}

\newtheorem{?}[theorem]{Problem}

\def\boxit#1{\leavevmode\hbox{\vrule\vtop{\vbox{\kern.33333pt\hrule
    \kern1pt\hbox{\kern1pt\vbox{#1}\kern1pt}}\kern1pt\hrule}\vrule}}

\usepackage{collectbox}



\newcommand{\f}[1]{\ifthenelse{\equal{#1}{1}}{(q;q)_\infty}{(q^{#1};q^{#1})_{\infty}}}

\begin{document}

\title[Some inequalities for $k$-colored partition functions]{Some inequalities for $k$-colored partition functions}

\author[S. Chern]{Shane Chern}
\address[Shane Chern]{Department of Mathematics, The Pennsylvania State University, University Park, PA 16802, USA}
\email{shanechern@psu.edu; chenxiaohang92@gmail.com}

\author[S. Fu]{Shishuo Fu}
\address[Shishuo Fu]{College of Mathematics and Statistics, Chongqing University, Huxi Campus LD506, Chongqing 401331, P.R. China}
\email{fsshuo@cqu.edu.cn}

\author[D. Tang]{Dazhao Tang}

\address[Dazhao Tang]{College of Mathematics and Statistics, Chongqing University, Huxi Campus LD206, Chongqing 401331, P.R. China}
\email{dazhaotang@sina.com}

\date{\today}

\begin{abstract}
Motivated by a partition inequality of Bessenrodt and Ono, we obtain analogous inequalities for $k$-colored partition functions $p_{-k}(n)$ for all $k\geq2$. This enables us to extend the $k$-colored partition function multiplicatively to a function on $k$-colored partitions, and characterize when it has a unique maximum. We conclude with one conjectural inequality that strengthens our results.
\end{abstract}



\maketitle

%\tableofcontents

%\tableofcontents

%%%%%%%%%%%%%%%%%%%%%%%%%%%%%%%%%%%%%
\section{Introduction}\label{sec1}
A \emph{partition} \cite{Andr1976} of a natural number $n$ is a finite weakly decreasing sequence of positive integers $\lambda_{1}\geq\lambda_{2}\geq\cdots\geq\lambda_{r}>0$ such that $\sum_{i=1}^{r}\lambda_{i}=n$. The $\lambda_{i}$'s are called the \emph{parts} of the partition. Let $p(n)$ enumerate the number of partitions of $n$. It is well known that
\begin{align}\label{ordi part func}
\sum_{n=0}^{\infty}p(n)q^{n}=\dfrac{1}{(q;q)_{\infty}}.
\end{align}

Here and throughout the paper, we use the following standard notation:
\begin{align*}
(a;q)_{\infty}=\prod_{n=0}^{\infty}(1-aq^{n}).
\end{align*}

Quite recently, Bessenrodt and Ono \cite{BO2016} defined a multiplicative function $p(\mu)$ on $P(n)$ as the product of parts of $\mu$, where $P(n)$ is the set of all partitions of $n$. Moreover, they proved that the maximal value of $p(\mu)$ is attained at a unique partition of $n$. Their proof relies on the following inequality for the partition function $p(n)$.

\begin{theorem}\label{part fun ineq}
For any integers $a,b$ such that $a,b>1$ and $a+b>9$, we have
\begin{align}\label{eq:BOinequality}
p(a)p(b)>p(a+b).
\end{align}
\end{theorem}

Many mathematicians have extended the work of Bessenrodt and Ono in an analogous way to other partition functions. For instance, Beckwith and Bessenrodt \cite{BB2016} studied the multiplicative properties of $k$-regular partitions. Hou and Jagadeesan \cite{HJ2017} also considered similar properties for Dyson's rank function $N(r,3;n)$, the number of partitions of $n$ with rank $\equiv r\pmod{3}$.

On the other hand, Alanazi, Gagola III and Munagi \cite{AGM2016} provided a combinatorial proof of Theorem \ref{part fun ineq}. Following the work of Bessenrodt and Ono as well as Alanazi et al. and relating to the second and third authors' recent work on $k$-colored partitions \cite{FT2017}, we consider this type of inequality for $k$-colored partition functions.

A partition is called a $k$-colored partition if each part can appear as $k$ colors. Let $p_{-k}(n)$ denote the number of $k$-colored partitions of $n$. The generating function of $p_{-k}(n)$ is given by
\begin{align}\label{k-gene func}
\sum_{n=0}^{\infty}p_{-k}(n)q^{n} &=\dfrac{1}{(q;q)_{\infty}^{k}}.
\end{align}

Interestingly, we obtain the following analogous inequality for $k$-colored partition functions.
\begin{theorem}\label{analog theorem}
For any positive integers $a\ge b$, we have
\begin{align}\label{k-colored part ineq}
p_{-k}(a)p_{-k}(b)>p_{-k}(a+b),
\end{align}
except for $(a,b,k)=(1,1,2), (2,1,2), (3,1,2), (1,1,3)$.
\end{theorem}
\begin{remark}\label{rmk}
Two remarks on Theorem \ref{analog theorem} are necessary. First, unlike the inequality \eqref{eq:BOinequality} for ordinary partition function $p(n)$, \eqref{k-colored part ineq} still holds even if one of $a$ and $b$ is $1$ (apart from the listed exceptions). On the other hand, for the triples $(a,b,k)=(2,1,2), (3,1,2), (1,1,3)$, we have equality in \eqref{k-colored part ineq}.
\end{remark}

Unlike the approach of Bessenrodt and Ono (and those who extended their work), who used effective estimates for coefficients of modular forms, our first proof of Theorem~\ref{analog theorem} uses elementary arguments together with the result of Bessenrodt and Ono. This is presented in Section~\ref{sect:1st proof}. A combinatorial proof is stated in the next section. In Section~\ref{sect:max property}, we obtain the explicit form of $k$-colored partitions with maximal property. We conclude in the last section with one conjecture that implies log-concavity and strong log-concavity for $k$-colored partitions.

\section{First proof of Theorem \ref{analog theorem}}\label{sect:1st proof}

\subsection{The case $k=2$}

%Without loss of generality, we assume $a\ge b$ in this section.

When $b=1$, we assume $a\ge 4$. For $a\le 3$, one may check directly. We first recall the following recurrence relation of $p(n)$ \cite[p. 12, Corollary 1.8]{Andr1976}:
$$p(n)=p(n-1)+p(n-2)-p(n-5)-p(n-7)+p(n-12)+p(n-15)-p(n-22)-\cdots,$$
which is due to Euler's pentagonal number theorem. This immediately tells that for $n\ge 2$
$$p(n-1)+p(n-2)\ge p(n).$$
It is also trivial to see that
$$2p(n-1)\ge p(n)$$
with equality holding only if $n=2$.

Now we have
\begin{align*}
p_{-2}(a)p_{-2}(1)-p_{-2}(a+1)&=2p_{-2}(a)-p_{-2}(a+1)\\
&=2\sum_{k=0}^a p(k)p(a-k)-\sum_{k=0}^{a+1}p(k)p(a+1-k)\\
&=\left(\sum_{k=0}^a p(k)\big(2p(a-k)-p(a+1-k)\big)\right)-p(a+1)\\
&\ge p(0)+p(1)+\cdots+p(a-2)+p(a)-p(a+1)\\
&> p(a-3)+p(a-2)+p(a)-p(a+1)\\
&\ge p(a-1)+p(a)-p(a+1)\\
&\ge 0.
\end{align*}
Hence we have that for $a\ge 4$,
$$p_{-2}(a)p_{-2}(1)>p_{-2}(a+1).$$

We now assume $b\ge 2$ and $a+b>20$. For the remaining cases, one may check directly. It follows by \eqref{ordi part func} and \eqref{k-gene func} that
\begin{align*}
p_{-2}(a)&=\sum_{\alpha_{1}+\alpha_{2}=a}p(\alpha_{1}) p(\alpha_{2}),\\
p_{-2}(b)&=\sum_{\beta_{1}+\beta_{2}=b}p(\beta_{1})p(\beta_{2}),\\
p_{-2}(a+b) &=\sum_{\gamma_{1}+\gamma_{2}=a+b}p(\gamma_{1})p(\gamma_{2}),
\end{align*}
where we always assume $\alpha_1,\alpha_2,\beta_1,\beta_2,\gamma_1,\gamma_2\in\mathbb{Z}_{\ge 0}$. Note that the first two identities give
\begin{align*}
p_{-2}(a)p_{-2}(b) =\sum_{\alpha_{1}+\alpha_{2}=a}\sum_{\beta_{1}+\beta_{2}=b}p(\alpha_{1})p(\alpha_{2})p(\beta_{1})p(\beta_{2}).
\end{align*}

The main idea of our proof is to find uniquely determined pairs $(\alpha_{1},\alpha_{2})$ and $(\beta_{1},\beta_{2})$ for a given $(\gamma_1,\gamma_2)$ with $\gamma_1+\gamma_2=a+b$ such that
\begin{align*}
\begin{cases}
\alpha_{1}+\alpha_{2}=a, \cr \beta_{1}+\beta_{2}=b, \cr \alpha_{1}+\beta_{1}=\gamma_{1}, \cr \alpha_{2}+\beta_{2}=\gamma_{2},
\end{cases}
\end{align*}
and
\begin{align*}
p(\alpha_{1})p(\alpha_{2})p(\beta_{1})p(\beta_{2})\geq p(\gamma_{1})p(\gamma_{2}).
\end{align*}
One readily notes that apart from the selections we are going to make, there exists some uncounted quadruples, say $(\alpha_1^*,\alpha_2^*,\beta_1^*,\beta_2^*)\in\mathbb{Z}_{\ge 0}^4$, with $\alpha_1^*+\alpha_2^*=a$ and $\beta_1^*+\beta_2^*=b$, while $p(\alpha_{1}^*)p(\alpha_{2}^*)p(\beta_{1}^*)p(\beta_{2}^*)$ contributes a positive value to $p_{-2}(a)p_{-2}(b)$. This immediately implies
\begin{align*}
p_{-2}(a)p_{-2}(b) &=\sum_{\alpha_{1}+\alpha_{2}=a}\sum_{\beta_{1}+\beta_{2}=b}p(\alpha_{1})p(\alpha_{2})p(\beta_{1})p(\beta_{2})\\
&> \sum_{\gamma_{1}+\gamma_{2}=a+b}p(\gamma_{1})p(\gamma_{2})\\
&=p_{-2}(a+b),
\end{align*}
which is our desired inequality.

\textit{Case 1}. Assume that $2\le b<a-2$. We first consider $\gamma_1$'s with $0\le \gamma_1\le (a+b)/2$. Now we choose
$$(\alpha_1,\alpha_2)=(\gamma_1,a-\gamma_1)\quad \text{and} \quad (\beta_1,\beta_2)=(0,b).$$
Note that $p(\alpha_1)p(\beta_1)=p(\gamma_1)$. Furthermore, $\alpha_2=a-\gamma_1\ge a-(a+b)/2=(a-b)/2>1$ while $\beta_2=b> 1$. Additionally, $\alpha_2+\beta_2=a-\gamma_1+b\ge (a+b)/2>10$. It follows by Theorem \ref{part fun ineq} that $p(\alpha_2)p(\beta_2)>p(\gamma_2)$. Hence
$$p(\alpha_{1})p(\alpha_{2})p(\beta_{1})p(\beta_{2})> p(\gamma_{1})p(\gamma_{2}).$$

When $(a+b)/2< \gamma_1\le a+b$, which is equivalent to $0\le \gamma_2< (a+b)/2$, we may choose
$$(\alpha_1,\alpha_2)=(a-\gamma_2,\gamma_2)\quad \text{and} \quad (\beta_1,\beta_2)=(b,0).$$
The rest of the argument is similar.

\textit{Case 2}. Assume that $a-2\le b\le a$. Since $a+b>20$, we have $a>10$ and $b>9$. Again we first consider $\gamma_1$'s with $0\le \gamma_1\le (a+b)/2$. This time $\gamma_1$ can be as large as $a$.

When $0\le \gamma_1<a-1$, we choose
$$(\alpha_1,\alpha_2)=(\gamma_1,a-\gamma_1)\quad \text{and} \quad (\beta_1,\beta_2)=(0,b).$$
Using the above argument, we readily have
$$p(\alpha_{1})p(\alpha_{2})p(\beta_{1})p(\beta_{2})> p(\gamma_{1})p(\gamma_{2}).$$

When $\gamma_1=a-1$, we choose
$$(\alpha_1,\alpha_2)=(\gamma_1-2,a-\gamma_1+2)\quad \text{and} \quad (\beta_1,\beta_2)=(2,b-2).$$
Note that $\alpha_1=\gamma_1-2=a-3>7$, $\beta_1=2$, and $\gamma_1=\alpha_1+\beta_1>9$, we have $p(\alpha_1)p(\beta_1)>p(\gamma_1)$. On the other hand, $\alpha_2=a-\gamma_1+2=3$, $\beta_2=b-2>7$, and $\gamma_2=\alpha_2+\beta_2>10$. We have $p(\alpha_2)p(\beta_2)>p(\gamma_2)$. Altogether,
$$p(\alpha_{1})p(\alpha_{2})p(\beta_{1})p(\beta_{2})> p(\gamma_{1})p(\gamma_{2}).$$

When $\gamma_1=a$ (this case happens only if $b=a$), we choose
$$(\alpha_1,\alpha_2)=(a,0)\quad \text{and} \quad (\beta_1,\beta_2)=(0,b).$$
Then
$$p(\alpha_{1})p(\alpha_{2})p(\beta_{1})p(\beta_{2})= p(\gamma_{1})p(\gamma_{2}).$$

For $(a+b)/2< \gamma_1\le a+b$, or equivalently $0\le \gamma_2< (a+b)/2$, the argument is similar. We omit the details here.

This completes our proof for $k=2$.

\subsection{The case $k\geq3$}

The proofs for $k=3$, $4$ and $5$ are similar to the proof for $k=2$ as we notice
\begin{align*}
p_{-k}(a)&=\sum_{\alpha_{1}+\alpha_{2}=a}p_{-2}(\alpha_{1}) p_{-(k-2)}(\alpha_{2}),\\
p_{-k}(b)&=\sum_{\beta_{1}+\beta_{2}=b}p_{-2}(\beta_{1})p_{-(k-2)}(\beta_{2}),\\
p_{-k}(a+b) &=\sum_{\gamma_{1}+\gamma_{2}=a+b}p_{-2}(\gamma_{1})p_{-(k-2)}(\gamma_{2}),
\end{align*}
where we adopt $p_{-1}(n)=p(n)$. Then we only need to subtly choose pairs $(\alpha_{1},\alpha_{2})$ and $(\beta_{1},\beta_{2})$ for each given $(\gamma_1,\gamma_2)$ with $\gamma_1+\gamma_2=a+b$ such that the pairs satisfy the same conditions as in the $k=2$ case. This argument is routine and hence we omit the details. However, when $k=3$, we should be careful when proving
$$p_{-3}(a)p_{-3}(1)>p_{-3}(a+1)$$
for $a\ge 2$. In fact, we have
\begin{align*}
p_{-3}(a)p_{-3}(1)-p_{-3}(a+1)&=3p_{-3}(a)-p_{-3}(a+1)\\
&=3\sum_{k=0}^a p_{-2}(k)p(a-k)-\sum_{k=0}^{a+1}p_{-2}(k)p(a+1-k)\\
&=\left(\sum_{k=0}^a p_{-2}(k)(3p(a-k)-p(a+1-k))\right)-p_{-2}(a+1)\\
&> 2p_{-2}(a)-p_{-2}(a+1)\\
&=p_{-2}(1)p_{-2}(a)-p_{-2}(a+1)\ge 0.
\end{align*}

Finally for $k>5$ we use induction. Here we simply use the fact that
$$p_{-k}(a)=\sum_{\alpha_{1}+\alpha_{2}=a}p_{-3}(\alpha_{1}) p_{-(k-3)}(\alpha_{2}),$$
while $k-3\ge 3$. Consequently, any quadruple $(\alpha_1,\alpha_2,\beta_1,\beta_2)\in\mathbb{Z}_{\ge 0}^4$ with
\begin{align*}
\begin{cases}
\alpha_{1}+\alpha_{2}=a, \cr \beta_{1}+\beta_{2}=b, \cr \alpha_{1}+\beta_{1}=\gamma_{1}, \cr \alpha_{2}+\beta_{2}=\gamma_{2},
\end{cases}
\end{align*}
gives us
$$p_{-3}(\alpha_{1})p_{-(k-3)}(\alpha_{2})p_{-3}(\beta_{1})p_{-(k-3)}(\beta_{2})\ge p_{-3}(\gamma_{1})p_{-(k-3)}(\gamma_{2}).$$ So for a fixed pair $(\gamma_1,\gamma_2)$, all quadruples $(\alpha_1,\alpha_2,\beta_1,\beta_2)$ are equally good for us. Then similar argument as for the case $k=2$ finishes the proof.

\section{Combinatorial viewpoint}\label{sect:comb proof}
Our first proof of Theorem~\ref{analog theorem} is analytic and inductive in nature. In this section, we present a combinatorial proof that is largely motivated by the work of Alanazi et al. \cite{AGM2016}.

We introduce some notations that will be used in the sequel. Firstly, following Andrews and Eriksson \cite{AE2004}, we denote the number of $k$-colored partitions of $n$ that satisfy a given condition by $p_{-k}(n|\textrm{ condition})$ while the enumerated set will be denoted by $P_{-k}(n|\textrm{ condition})$.

Secondly, the Cartesian product of two sets of partitions $A$ and $B$, denoted as $A\oplus B$, is given by
\begin{align*}
A\oplus B=\{(\lambda;\mu)|\lambda\in A,\mu\in B\}.
\end{align*}

Finally, for $k$-colored partitions, we use the subscript $1,2,\cdots,k$ to indicate the color of a specific part, while the superscript refers to multiplicities. For example, $3_{5}^2$ stands for two parts of size 3 with color 5. And we always arrange the parts in a unique way such that both their sizes and colors are weakly decreasing. For instance, $(4_2,2_3^2,2_1,1_2,1_1)$ is a $3$-colored partition of $12$ written in the standard way. To avoid heavy notation and clarify possible confusion, we point out that the meaning of $\lambda_i$ is the $i$-th part of a partition $\lambda$, not a part of size $\lambda$ and color $i$.

In order to prove the theorem, we first establish three lemmas.
\begin{lemma}\label{key lemma}
If $c\ge d$ are positive integers and $k\geq2$, then
\begin{align}\label{ine:cd}
p_{-k}(c|~\emph{no}~1_{1}'\emph{s})p_{-k}(d|~\emph{no}~1_{2}'\emph{s})\geq p_{-k}(c+d|~\emph{no}~1_{1}'\emph{s}~\emph{and}~\emph{no}~1_{2}'\emph{s}),
\end{align}
except for $(c,d,k)=(1,1,2)$.
\end{lemma}
\begin{proof}
Firstly for the cases $(c,d,k)=(1,1,k)$, we compute directly
\begin{align*}
p_{-k}(1|~\text{no}~1_{1}'\text{s})p_{-k}(1|~\text{no}~1_{2}'\text{s})&=(k-1)^2,\\
p_{-k}(2|~\text{no}~1_{1}'\text{s}~\text{and}~\text{no}~1_{2}'\text{s})&=k+(k-2)+\cdots+1=\frac{k(k-1)}{2}+1.
\end{align*}
So \eqref{ine:cd} holds for $k\ge 3$. Next we assume $c\ge 2$.
For $\lambda=(\lambda_{1},\lambda_{2},\cdots,\lambda_{t})\in P_{-k}(c+d)$, let
\begin{align*}
i=i(\lambda)=\max\{j\in\mathbb{N}|~1\leq j\leq t, \lambda_{j}+\cdots+\lambda_{t}\geq d\}.
\end{align*}
Moreover, let $\lambda_{i}=x+y$ $\left(x=x(\lambda),y=y(\lambda)\right)$ such that
\begin{align*}
x+\lambda_{i+1}+\cdots+\lambda_{t}=d\quad \textrm{and}\quad y+\lambda_{1}+\cdots+\lambda_{i-1}=c.
\end{align*}
Notice that $0<x\leq \lambda_{i}$. Now define a map
\begin{align*}
f_{k}: P_{-k}(c+d|~\textrm{no}~1_{1}'\textrm{s}~\textrm{and}~\textrm{no}~1_{2}'\textrm{s})\rightarrow P_{-k}(c|~\textrm{no}~1_1'\textrm{s})\oplus P_{-k}(d|~\textrm{no}~1_{2}'\textrm{s}),
\end{align*}
as follows. For $\lambda=(\lambda_{1},\lambda_{2},\cdots,\lambda_{t})\in P_{-k}(c+d|~\textrm{no}~1_{1}'\textrm{s}~\textrm{and}~\textrm{no}~1_{2}'\textrm{s})$,
\begin{align*}
f_{k}(\lambda):=
\begin{cases}
(\lambda_{1},\lambda_{2},\cdots,\lambda_{i-1};~\lambda_{i},\cdots,\lambda_{t}), &\textrm{if}~y=0; \cr (\lambda_{1},\lambda_{2},\cdots,\lambda_{i-1},y_{c(\lambda_{i})};~
\lambda_{i+1},\cdots,\lambda_{t},1_{1}^{x}), &\textrm{if}~y\geq2; \cr (\lambda_{1},\lambda_{2},\cdots,\lambda_{i-2},1_{2}^{\lambda_{i-1}+1};~\lambda_{i+1},\cdots,
\lambda_{t},1_{1}^{x}), &\textrm{if}~y=1,c(\lambda_i)=1; \cr (\lambda_{1},\lambda_{2},\cdots,\lambda_{i-1},1_{c(\lambda_i)};~\lambda_{i+1},\cdots,\lambda_{t},1_{1}^{x}), &\textrm{if}~y=1,c(\lambda_i)\ne 1;
\end{cases}
\end{align*}
where $c(\lambda_{i})$ denotes the color of part $\lambda_{i}$. Since $c\ge 2$, in the case of $y=1,c(\lambda_i)=1$, we must have $\lambda_{i-1}\ge \lambda_i=x+y\ge 2$, so the images of the third and fourth cases are disjoint, therefore $f_k$ is indeed one-to-one, and we are done.
\end{proof}

\begin{lemma}\label{lemma1}
If $a$ is a positive integer and $k\geq2$, then
\begin{align}\label{induct ineq}
p_{-k}(a|~\emph{no}~1_{1}'\emph{s})p_{-k}(1) >p_{-k}(a+1|~\emph{no}~1_{1}'\emph{s}),
\end{align}
except for $(a,k)=(1,2), (1,3)$.
\end{lemma}

\begin{proof}
Firstly for $a=1$, we simply compute $p_{-k}(1|~\text{no}~1_{1}'\text{s})p_{-k}(1)=(k-1)k$, while $p_{-k}(2|~\text{no}~1_{1}'\text{s})=k+(k-1)+\cdots+1=k(k+1)/2$, and clearly \eqref{induct ineq} holds for $k\ge 4$. Now assume $a\ge 2$.
Given a partition $\lambda=(\lambda_{1},\lambda_{2},\cdots,\lambda_{t})\in P_{-k}(a+1|~\text{no}~1_{1}'\text{s})$, where $\lambda_{t}\neq1_{1}$. We define a map
\begin{align*}
g_{k}: P_{-k}(a+1|~\textrm{no}~1_{1}'\textrm{s})\rightarrow P_{-k}(a|~\textrm{no}~1_{1}'\textrm{s})\oplus P_{-k}(1)
\end{align*}
by
\begin{align*}
g_{k}(\lambda):=
\begin{cases}
(\lambda_{1},\lambda_{2},\cdots,\lambda_{t-1}+(\lambda_{t}-1)_1;~1_{1}), &\textrm{if}~\lambda_{t}\geq3; \cr(\lambda_{1},\lambda_{2},\cdots,\lambda_{t-2},1_{2}^{\lambda_{t-1}+1};~1_{1}), &\textrm{if}~\lambda_{t}=2_{1}; \cr (\lambda_{1},\lambda_{2},\cdots,\lambda_{t-1},1_{c};~1_{1}), &\textrm{if}~\lambda_{t}=2_{c}, c\neq1; \cr (\lambda_{1},\lambda_{2},\cdots,\lambda_{t-1};~1_{c}) &\textrm{if}~\lambda_{t}=1_{c}, c\neq1;
\end{cases}
\end{align*}

Since $a\ge 2$, in the case of $\lambda_t=2_1$, we must have $\lambda_{t-1}\ge 2$, so the images of the second and third cases are disjoint. And the pair $(\cdots, 1_2^2;~ 1_1)$ is not in the image of $g_k$. Hence $g_{k}$ is one-to-one but not onto, which gives
\begin{align*}
p_{-k}(a+1|~\textrm{no}~1_{1}'\textrm{s})< p_{-k}(a|~\textrm{no}~1_{1}'\textrm{s})p_{-k}(1),
\end{align*}
as required.
\end{proof}

\begin{lemma}\label{lemma2}
If $a, b$ are integers with $a\geq b\geq1$, then
\begin{align*}
p_{-k}(a|~\emph{no}~1_{1}'\emph{s})p_{-k}(b) >p_{-k}(a+b|~\emph{no}~1_{1}'\emph{s}),
\end{align*}
except for $(a,b,k)=(1,1,2), (1,1,3)$.
\end{lemma}
\begin{proof}
For $b=1$ this reduces to Lemma~\ref{lemma1}. So we can assume $a\ge b\ge 2$. Let $n=a+b$. We will apply induction on $n$.

Base case ($a=b=2,n=4$): With some effort one computes that $$p_{-k}(2|~\text{no}~1_{1}'\text{s})p_{-k}(2) - p_{-k}(4|~\text{no}~1_{1}'\text{s})=5\binom{k+2}{4}>0, \text{ for }k\ge 2.$$

Inductive step: Suppose the inequality is true for smaller sum $a+b$. Then by the inductive hypothesis, Lemma~\ref{key lemma} and Lemma \ref{lemma1}, we obtain
\begin{align*}
 &p_{-k}(a+b|~\textrm{no}~1_{1}'\textrm{s})\\
  =&p_{-k}(a+b|~\textrm{no}~1_{1}'\textrm{s~and~at~least~one}~1_{2}'\textrm{s})+p_{-k}(a+b|~\textrm{no}~1_{1}'\textrm{s~and~no~}1_{2}'\textrm{s})\\
 =&p_{-k}(a+b-1|~\textrm{no}~1_{1}'\textrm{s})+p_{-k}(a+b|~\textrm{no}~1_{1}'\textrm{s~and~no~}1_{2}'\textrm{s})\\
 <&p_{-k}(a|~\textrm{no}~1_{1}'\textrm{s})p_{-k}(b-1)+p_{-k}(a+b|~\textrm{no}~1_{1}'\textrm{s~and~no~}1_{2}'\textrm{s})\\
 \le &p_{-k}(a|~\textrm{no}~1_{1}'\textrm{s})p_{-k}(b|\textrm{~at~least~one~}1_{2}'\textrm{s})+p_{-k}(a|~\textrm{no}~1_{1}'\textrm{s})p_{-k}(b|~\textrm{no}~1_{2}'\textrm{s})\\
 =&p_{-k}(a|~\textrm{no}~1_{1}'\emph{s})p_{-k}(b).
\end{align*}
Hence, by the principle of mathematical induction, the inequality holds for all $n\geq4$. This finishes the proof.
\end{proof}

Now, we are ready to prove Theorem \ref{analog theorem} by induction.
\begin{proof}[2nd proof of Theorem \ref{analog theorem}]
Let $n=a+b$. We apply induction on $n$.

Base case: It can be verified that the inequality holds for $n=5$ with any $k\ge 2$.

Inductive step: Assume that $n\geq 6$ and the inequality holds for $n-1$. Without loss, suppose $a\geq b$, so $a\ge 3$. Thus, by inductive hypothesis, $p_{-k}(a+b-1)<p_{-k}(a-1)p_{-k}(b)$. According to Lemma~\ref{lemma2}, $p_{-k}(a+b|~\textrm{no}~1_{1}'\textrm{s})<p_{-k}(a|~\textrm{no}~1_{1}'\textrm{s})p_{-k}(b)$, Therefore,
\begin{align*}
p_{-k}(a+b) &=p_{-k}(a+b|~\textrm{at~least~one}~1_{1}'\textrm{s})+p_{-k}(a+b|~\textrm{no}~1_{1}'\textrm{s})\\
 &=p_{-k}(a+b-1)+p_{-k}(a+b|~\textrm{no}~1_{1}'\textrm{s})\\
 &<p_{-k}(a-1)p_{-k}(b)+p_{-k}(a|~\textrm{no}~1_{1}'\textrm{s})p_{-k}(b)\\
 &=p_{-k}(a-1)p_{-k}(b)+\{p_{-k}(a)-p_{-k}(a|~\textrm{at~least~one}~1_{1}'\textrm{s})\}p_{-k}(b)\\
 &=p_{-k}(a-1)p_{-k}(b)+\{p_{-k}(a)-p_{-k}(a-1)\}p_{-k}(b)\\
 &=p_{-k}(a)p_{-k}(b).
\end{align*}
Hence, by the principle of mathematical induction, the inequality holds for $n\geq5, k\ge 2$. The rest of cases including the exceptions can be individually checked. This completes the proof.
\end{proof}

\section{The maximal property}\label{sect:max property}
Similar to Bessenrodt and Ono \cite{BO2016}, we define an extended partition function for $k$-colored partitions as
\begin{align*}
p_{-k}(\lambda)=\prod_{j\geq1}p_{-k}(\lambda_{j}).
\end{align*}
Moreover, we will consider the the following maximal value
\begin{align*}
\max p_{-k}(n)=\max \left(p_{-k}(\lambda)|~\lambda\in P_{-k}(n)\right).
\end{align*}

\begin{theorem}\label{maximal theorem}
Let $n\in\mathbb{N}$. For $n\geq1$, the maximal value $\max p_{-k}(n)$ of the extended partition function on $P_{-k}(n)$ is given by (type of size)
\begin{enumerate}[1)]
\item $k=2$:
\begin{align*}
\begin{cases}
(2,2,\cdots,2),\ &\emph{when}~n\equiv0\pmod{2}, \cr (3,2,\cdots,2)~\emph{or}~(2,2,\cdots,2,1),\ &\emph{when}~n\equiv1\pmod{2}.
\end{cases}
\end{align*}
Then
\begin{align*}
\max p_{-2}(n)=
\begin{cases}
5^{\frac{n}{2}},\ &\emph{if}~n\equiv0\pmod{2}, \cr 2\cdot5^{\frac{n-1}{2}},\ &\emph{if}~n\equiv1\pmod{2}.
\end{cases}
\end{align*}

\item $k=3$:
\begin{align*}
(\underbrace{2,\cdots,2}_{m},\underbrace{1,\cdots,1}_{l}), \emph{ such that }2m+l=n.
\end{align*}
Then
\begin{align*}
\max p_{-3}(n) &=\left(p_{-3}(1)\right)^{n}=3^{n}.
\end{align*}

\item $k\geq 4$:
\begin{align*}
(1,1,\cdots,1).
\end{align*}
Then
\begin{align*}
\max p_{-k}(n) &=\left(p_{-k}(1)\right)^{n}=k^{n}.
\end{align*}
\end{enumerate}
\end{theorem}

\begin{proof}
We will only prove the case $k=2$, the other cases are similar and indeed simpler. First note that thanks to Theorem~\ref{k-colored part ineq}, for any part $\mu_i\geq 4$, replacing it by the two parts $\lceil\frac{\mu_{i}}{2}\rceil$, $\lfloor\frac{\mu_{i}}{2}\rfloor$ results in a larger value for $p_{-k}(\mu)$. So in a partition $\mu$ that achieves the maximal value, all parts should be smaller than $3$ in size. Then we utilize the first row of Table~\ref{Tab2} as well as Remark~\ref{rmk} on the exceptions to see that there can be at most one $3$ or at most one $1$ in $\mu$ but not both. For example, two $3$'s is inferior to three $2$'s since $p_{-2}(3)^2=100<p_{-2}(2)^3=125$. Once the partition type is determined, the formula for $\max p_{-2}(n)$ follows.
\end{proof}

\section{Final remarks}
We conclude with some questions and remarks to motivate further investigation.

A sequence $\{a_{n}\}_{n=1}^{\infty}$ is called \emph{log-concave} if it fulfills
\begin{align*}
a_{n}^{2}-a_{n-1}a_{n+1}\geq0\quad \textrm{for~all}~n\ge 2.
\end{align*}

In \cite{DP2015}, DeSalvo and Pak showed that
\begin{theorem}[Theorem 1.1, \cite{DP2015}]
The sequence $\{p(n)\}$ is log-concave for all $n>25$.
\end{theorem}
Furthermore, they proved the following strong log-concavity for partition function $p(n)$.
\begin{theorem}[Theorem 5.1, \cite{DP2015}]
For all $n>m>1$, we have
\begin{align*}
p(n)^{2}>p(n-m)p(n+m).
\end{align*}
\end{theorem}
The numerical evidence (see Table \ref{Tab2}) suggests the following conjecture.

\begin{conjecture}
For all $n>m\geq 1$ and $k\geq2$, except for $(k,n,m)=(2,6,4)$, we have
\begin{align*}
p_{-k}(n-1)p_{-k}(m+1)\geq p_{-k}(n)p_{-k}(m).
\end{align*}
\end{conjecture}
\begin{table}[tbp]\caption{A Table of values of $p_{-k}(n)$}\label{Tab2}
\centering
\begin{tabular}{||c|ccccccccccc||}
\hline
$k\setminus n$ &1 &2 &3 &4 &5 &6 &7 &8 &9 &10 &11\\
\hline
2 &2 &5 &10 &20 &36 &65 &110 &185 &300 &481 &752\\
3 &3 &9 &22 &51 &108 &221 &429 &810 &1479 &2640 &4599\\
4 &4 &14 &40 &105 &252 &574 &1240 &2580 &5180 &10108 &19208\\
5 &5 &20 &65 &190 &506 &1265 &2990 &6765 &14725 &31027 &63505\\
6 &6 &27 &98 &315 &918 &2492 &6372 &15525 &36280 &81816 &178794\\
7 &7 &35 &140 &490 &1547 &4522 &12405 &32305 &80465 &192899 &447146\\
8 &8 &44 &192 &726 &2464 &7704 &22528 &62337 &164560 &417140 &1020416\\
9 &9 &54 &255 &1035 &3753 &12483 &38709 &113265 &315445 &841842 &2164185\\
10 &10 &65 &330 &1430 &5512 &19415 &63570 &195910 &573430 &1605340 &4322110\\
\hline
\end{tabular}
\end{table}

In a recent paper, the second and third authors \cite{FT2017} introduced a generalized crank $M_{k}(m,n)$ for $k$-colored partition functions. It would be appealing to find some inequalities for $M_{k}(m,n)$ analog to Dyson's rank function $N(r,3;n)$.

\section*{Acknowledgement}
We are indebted to the anonymous referee whose helpful suggestions and comments have made the first section more complete. The second and third authors were supported by the National Natural Science Foundation of China (No.~11501061).

\end{document}